\documentclass[12pt]{article}
\usepackage[T1]{fontenc}
\usepackage{a4wide}
\usepackage{lipsum}
\usepackage{ textcomp }
\usepackage[inline,shortlabels]{enumitem}

\usepackage{todonotes}

\title{Extremal Problems on Forest Cuts and \\ Acyclic Neighborhoods in Sparse Graphs\footnote{\scriptsize This work started during the 6th WoPOCA (Workshop Paulista em Otimização, Combinatória e Algoritmos) in Campinas, Brazil. We thank the organizers for the opportunity to start new collaborations and the agencies that helped making the workshop possible: Conselho
Nacional de Desenvolvimento Científico e Tecnológico -- CNPq (404315/2023-2) and FAEPEX (2422/23).
F.~Botler would also like to thank the support of CNPq (304315/2022-2), CAPES (88887.878880/2023-00) and the São Paulo Research Foundation -- FAPESP (2024/14906-6).
Y.S.~Couto would also like to thank the support of FAPESP (2024/18049-0).
C.G.~Fernandes would also like to thank the support of CNPq (310979/2020-0).
V.F.~dos Santos would also like to thank the support of CNPq (312069/2021-9, 406036/2021-7 and 404479/2023-5) and FAPEMIG (APQ-01707-21). C.M.~Sato would also like to thank the support of CNPq (408180/2023-4).
E.F.~de Figueiredo would also like to thank the support of CAPES (Finance Code 001).
}
}

\author{
F.~Botler\thanks{\scriptsize Institute of Mathematics and Statistics, University of São Paulo, Brazil. E-mail: {\tt \{fbotler,cris,yancouto\}@ime.usp.br}.} \and
Y.~S.~Couto$^\dagger$ \and
C.~G.~Fernandes$^\dagger$ \and
E.~F.~de~Figueiredo\thanks{\scriptsize Department of Computer Science, Federal University of Minas Gerais, Brazil. E-mail: {\tt \{ederfgd,viniciussantos\} @dcc.ufmg.br}.} \and
R.~Gómez\thanks{\scriptsize Center for Mathematics, Computing and Cognition, Federal University of ABC, Brazil. E-mail: {\tt \{gomez.renzo,c.sato\} @ufabc.edu.br}.} \and
V.~F.~dos~Santos$^\ddagger$ \and
C.~M.~Sato$^\S$
}

\makeatletter

\newcommand{\shorttitle}{Extremal Problems on Forest Cuts and Acyclic Neighborhoods in Sparse Graphs}

\def\@maketitle{%
  \newpage
  \begin{center}%
  \let \footnote \thanks
    {\small Proceedings of the 13th European Conference on Combinatorics, Graph Theory and Applications\\ EUROCOMB'25\\
    Budapest, August 25 - August 29, 2025
    }
    \vskip 0.5em
    \rule{\linewidth}{0.04cm}
    \vskip 3.2em
    {\LARGE \textbf{\textsc{\@title}} \par}%
    \vskip 1.2em
    {\textbf{\textsc{(Extended abstract)}} \par}
    \vskip 2.0em%
    {\large
      \lineskip .5em%
      \begin{tabular}[t]{c}%
        \@author
      \end{tabular}\par}%
  \end{center}%
  \par
  }
\makeatother

\usepackage{comment}
\usepackage{latexsym} 
\usepackage[english]{babel}
\usepackage{float}
\usepackage{amsfonts,amsmath,amssymb,amsbsy,amsthm,enumerate}
\usepackage[mathscr]{euscript}
\usepackage{pstricks}
\usepackage{multirow}
\usepackage{verbatim}
\usepackage{color}
\usepackage{enumerate}
\usepackage[normalem]{ulem} % para usar \sout
\usepackage{hyperref}

\newtheorem{theorem}{Theorem}
\newtheorem{corollary}[theorem]{Corollary}
\newtheorem{conjecture}[theorem]{Conjecture}
\newtheorem{lemma}[theorem]{Lemma}

\newtheorem{remark}[theorem]{Remark}

\usepackage{graphicx} % Required for inserting images
\usepackage{comment}
\usepackage{cleveref}
\usepackage{tikz}
\usepackage{subcaption}
\usetikzlibrary{calc,positioning,decorations.pathmorphing,decorations.pathreplacing}
\usetikzlibrary{backgrounds}
\usepackage{./Figures/figstyle}
\usepackage[normalem]{ulem}

\newcommand*\oline[1]{%
  \vbox{%
    \hrule height 0.5pt%                  % Line above with certain width
    \kern0.25ex%                          % Distance between line and content
    \hbox{%
      \kern-0.1em%                        % Distance between content and left side of box, negative values for lines shorter than content
      \ifmmode#1\else\ensuremath{#1}\fi%  % The content, typeset in dependence of mode
      \kern0.1em%                        % Distance between content and left side of box, negative values for lines shorter than content
    }% end of hbox
  }% end of vbox
}

\usepackage{tikz}
\usepackage[subrefformat=parens,labelformat=parens]{subcaption}
\usetikzlibrary{calc,positioning,decorations.pathmorphing,decorations.pathreplacing}
\tikzset{black vertex/.style={circle,draw,minimum size=1mm,inner sep=0pt,outer sep=2pt,fill=black, color=black}}

\begin{document}

\thispagestyle{empty}
\maketitle

\begin{abstract}
Chernyshev, Rauch, and Rautenbach proved that every connected graph $G$ on $n$ vertices for which $e(G) < \frac{11}5n-\frac{18}5$ has a vertex cut that induces a forest, and conjectured that the same remains true if $e(G) < 3n-6$ edges. 
We improve their result by proving that every connected graph on $n$ vertices for which $ e(G) < \frac94n-\frac{15}4$ has a vertex cut that induces a forest. We also study weaker versions of the problem that might lead to an improvement on the bound obtained. 

% Keywords:
% graph theory
% forest cut
% connectivity
% acyclic neighborhood

\vspace{2mm}
  {\noindent\small\bf DOI:}  {\tt https://doi.org/10.5817/CZ.MUNI.EUROCOMB25-000}
\end{abstract}

%%%%%%%%%%%%%%%%%%%%%%%%%%%%%%%%%%%%%%%%%%%%%%%%%%%%%%%%%%%%%%%%%%%%%%%%%%%%
%%%%%%%%%%%%%%%%%%%%%%%%%%%%%%%%%%%%%%%%%%%%%%%%%%%%%%%%%%%%%%%%%%%%%%%%%%%%
%%%%%%%%%%%%%%%%%%%%%%%%%%%%%%%%%%%%%%%%%%%%%%%%%%%%%%%%%%%%%%%%%%%%%%%%%%%%
\section{Introduction}

Let~$G$ be a connected graph. A set $S \subset V(G)$ is a \emph{vertex cut} if $G-S$ is disconnected. If $|S|=k$, we say $S$ is a \emph{$k$-vertex cut}. If $S$ is an independent set, we say $S$ is an \emph{independent cut}. Vertex cuts with special properties have been studied in different contexts. Chen and Yu~\cite{chen-yu} showed that every connected graph with less than $2n-3$ edges has an independent cut, confirming a conjecture due to Caro.
Recently, Chernyshev, Rauch, and Rautenbach proposed the following analogue conjecture, replacing independent set by forest~\cite[Conjecture~1]{chernyshev}.  A \emph{forest cut} is a vertex cut that induces a forest.
%They conjectured the following. 

\begin{conjecture}[Chernyshev--Rauch--Rautenbach, 2024]
\label{conj:3n}
  If~$G$ is a connected graph on $n$ vertices with no forest cut, then $e(G) \geq 3n-6$.
\end{conjecture}

Chernyshev et al.~\cite{chernyshev} also showed that Conjecture~\ref{conj:3n} holds for some classes of graphs. 
For instance, they showed that a graph \(G\) with \(n\) vertices has a forest cut if
(i) \(G\) is a planar graph that is not triangulated; 
(ii) \(G\) has a universal vertex and \(e(G) < 3n-6\); or
(iii) \(G\) is connected and \(e(G) < \frac{11}{5}n-\frac{18}{5}\).
% any planar graph that is not triangulated has a forest cut, that any graph with a universal vertex and less than $3n-6$ edges has a forest cut, and that connected graphs with less than $\frac{11}{5}n-\frac{18}{5}$ edges have a forest cut. 

We say a graph is \emph{$k$-cyclic} if every vertex set of size at most $k$ is dominating or has a cycle in its neighborhood.  
% \textcolor{red}{For instance, every triangulated planar graph is 2-cyclic.}
Note that any (forest) cut disconnects the graph into at least two components, which are not dominating sets, 
and one of these components has less than~$n/2$ vertices.
So, Conjecture~\ref{conj:3n} claims that any~$(\frac{n-1}2)$-cyclic graph has at least~$3n-6$ edges.
Moreover, any 2-vertex cut is trivially a forest, so Chernyshev et al.~\cite{chernyshev} noted that finding good lower bounds for the number of edges on 1-cyclic 3-connected graphs would imply a result towards Conjecture~\ref{conj:3n}, and stated the following.
% Moreover, any 2-vertex cut is trivially a forest, and in fact Chernyshev et al.~\cite{chernyshev} showed that any minimal counterexample for Conjecture 1 is in fact 4-connected, noted that finding good lower bounds for the number of edges on 1-cyclic 3-connected graphs would imply a result towards Conjecture~\ref{conj:3n}, and stated the following.

\begin{conjecture}[Chernyshev--Rauch--Rautenbach, 2024]
\label{conj:construction}
  If~$G$ is a 3-connected graph on $n$ vertices such that there is a cycle in the neighborhood of every vertex, then $e(G) \geq \frac73n-\frac73$.
\end{conjecture}

The conjecture addresses a proper subclass of \(1\)-cyclic graphs as it requires cycles in the neighborhood of universal vertices. However, it is functionally the same as for 1-cyclic graphs, as even Conjecture~\ref{conj:3n} holds for graphs with universal vertices~\cite{chernyshev}.
% For instance, $K_4-e$ is 1-cyclic and does not satisfy this requirement.
In this paper, we
improve the bound from~\cite{chernyshev} towards Conjecture~\ref{conj:3n},
disprove Conjecture~\ref{conj:construction}, and
present lower bounds on the number of edges for 3-connected graphs to be 1-cyclic and~2-cyclic.
% Também 4-connected né, mas a gente não menciona

\begin{comment}
\begin{theorem}\label{thm:9-4}
   If~$G$ is a connected graph on $n$ vertices and $m$ edges with no forest cut, then $m \geq \frac{9n-15}{4}$.
\end{theorem}

\begin{theorem}\label{thm:15-8}
  If~$G$ is a 3-connected 1-cyclic graph on $n \geq 6$ vertices and $m$ edges, then $m \geq \frac{15n}8$.
\end{theorem}

\begin{theorem}\label{thm:2n}
  If~$G$ is a 3-connected 2-cyclic graph on $n \geq 6$ vertices and $m$ edges, then $m \geq 2n$.
\end{theorem}
\end{comment}

\begin{theorem}\label{thm:teo}
  Let \(G\) be a graph on \(n\) vertices.
  Then the following hold.
    \begin{enumerate*}[label=(\alph*),itemjoin={\ }]
    \item\label{thm:9-4}If \(G\) is connected and has no forest cut, then $e(G) \geq \frac94n-\frac{15}4$;
    \item\label{thm:15-8}If \(G\) is \(3\)-connected, \(1\)-cyclic, and \(n\geq 6\), then $e(G) \geq \frac{15}8n$;
    \item\label{thm:2n}If \(G\) is \(3\)-connected, \(2\)-cyclic, and \(n\geq 6\), then $e(G) \geq 2n$.
    \end{enumerate*}
\end{theorem}

The $n \geq 6$ is necessary in Theorem~\ref{thm:teo}\ref{thm:15-8} and~\ref{thm:teo}\ref{thm:2n} as $K_5$ minus an edge is 3-connected and 2-cyclic (hence also 1-cyclic), has five vertices and nine edges, but $9 < \frac{15}8 \cdot 5 = \frac{75}8 < 10$.

\begin{remark}\label{rem:todas} There are infinite families of
\begin{enumerate*}[label=(\alph*),itemjoin={\ }]
\item\label{rm:3conn-1cyc}3-connected 1-cyclic graphs on~$n$ vertices with exactly~$\frac{15n}{8}$ edges and no universal vertices;
\item\label{rm:4conn-1cyc}4-connected 1-cyclic graphs on~$n$ vertices with exactly~$2n$ edges;
\item\label{rm:3conn-2cyc}3-connected 2-cyclic graphs on~$n$ vertices with exactly~$\frac94n$ edges;
\item\label{rm:4conn-2cyc}4-connected 2-cyclic graphs on~$n$ vertices with exactly~$\frac73n$ edges.
\end{enumerate*}
\end{remark}

Remark~\ref{rem:todas}\ref{rm:3conn-1cyc} disproves Conjecture~\ref{conj:construction}, proving that Theorem~\ref{thm:teo}\ref{thm:15-8} is asymptotically tight.
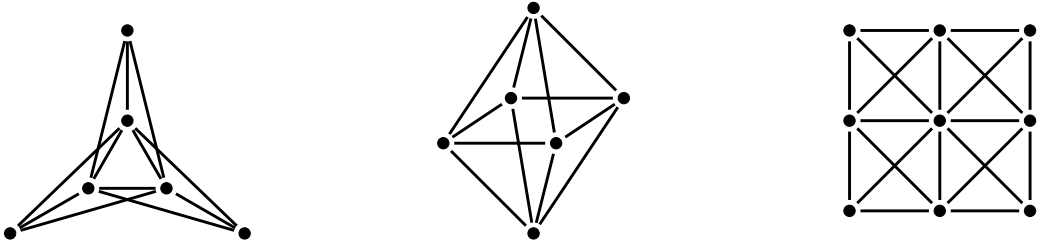
For Theorem~\ref{thm:teo}\ref{thm:2n}, we present a 3-connected 2-cyclic graph and a 4-connected 2-cyclic graph, both with 6 vertices and 12 edges, and, based on Remark~\ref{rem:todas}\ref{rm:4conn-2cyc}, we pose the following conjecture that would imply an improvement on Theorem~\ref{thm:teo}\ref{thm:9-4}, towards Conjecture~\ref{conj:3n}.

\begin{conjecture}\label{conj:73}
  If~$G$ is a 4-connected 2-cyclic graph on $n \geq 9$ vertices, then $e(G) \geq \frac73n$.
\end{conjecture}

In Section~\ref{sec:forestcut}, we prove Theorem~\ref{thm:teo}\ref{thm:9-4}.  
In Section~\ref{sec:1cyclic}, we prove Theorem~\ref{thm:teo}\ref{thm:15-8}-\ref{thm:2n}, 
and Remark~\ref{rem:todas}.
A recent independent work by Li, Tang, and Zhan~\cite{li2024minimum} contains results similar to the ones on 1-cyclic graphs in Section~\ref{sec:1cyclic}.
% (also disproving Conjecture~\ref{conj:construction}).
Due to space constraints, we omit a few proofs.

\section{Avoiding forest cuts}\label{sec:forestcut} 

% In this section, we prove \Cref{thm:9-4}. 
Chernyshev et al.~\cite{chernyshev} proved that a connected graph on $n$ vertices with no forest cut must have at least $\frac{11n}5-\frac{18}5$ edges. 
% In their proof, they present properties that must be satisfied by any  counterexample with a minimum number of vertices. 
For that, they studied properties its counterexamples with a minimum number of vertices.
Such properties are in fact shared with a minimum counterexample to Theorem~\ref{thm:teo}\ref{thm:9-4} and Conjecture~\ref{conj:3n}. To help the exposition, we state a conjecture parameterized by a number $\alpha$ with $2 \leq \alpha \leq 3$. 

\begin{conjecture}[$\alpha$-FC Conjecture]\label{conj:alpha}
  \emph{If~$G$ is a connected graph on $n$ vertices with no forest cut, then $e(G) \geq \alpha (n - 3) + 3$.}
\end{conjecture}

Note that Theorem~\ref{thm:teo}\ref{thm:9-4} is the same as the $\frac94$-FC Conjecture, Chernyshev et al.~\cite{chernyshev} proved the $\frac{11}5$-FC Conjecture and Conjecture~\ref{conj:3n} is the same as the $3$-FC Conjecture. 
For $2 \leq \alpha \leq 3$, a minimum counterexample to the $\alpha$-FC Conjecture is a graph \(G\) on $n$ vertices with no forest cut, ${e(G) < \alpha(n-3) + 3}$ and $n$ as small as possible.
The following lemma is used in the proof of Theorem~\ref{thm:teo}\ref{thm:9-4}.

\begin{lemma}\label{lem:todos}
  Let~$G$ be a minimum counterexample to the $\alpha$-FC Conjecture, for $2 \leq \alpha \leq 3$. Then
  \begin{enumerate*}[label=(\alph*),itemjoin={\ }]
  \item\label{lem:4connected}$G$ is 4-connected and has at least $8$ vertices;
  \item\label{lem:noC4}no degree-4 vertex in~$G$ has a $C_4$ in its neighborhood; and
  \item\label{lem:nod4insameK4}no two degree-4 vertices are in the same~$K_4$ in~$G$.
  \end{enumerate*}
\end{lemma}

Lemma~\ref{lem:todos}\ref{lem:4connected}  was adapted from the proof of Claim~1 in Chernyshev et al.~\cite{chernyshev}.
They~\cite[Claim~2]{chernyshev} also proved that, in a minimum counterexample to Conjecture~\ref{conj:3n}, every degree-4 vertex has at most two neighbors of degree 4. 
Lemma~\ref{lem:todos}\ref{lem:noC4} and~\ref{lem:todos}\ref{lem:nod4insameK4} are strengthenings of this statement. 
Lemma~\ref{lem:todos}\ref{lem:noC4} implies that every degree-4 vertex in a minimum counterexample to the $\alpha$-FC Conjecture lies in a $K_4$, and we deduce the following   
% In what follows, we denote by $N_G(v)$ the set of neighbors of a vertex $v$ in a graph~$G$. If~$G$ is clear from the context, we write simply $N(v)$. Similarly, $N(S)$ is the set of neighbors of a set $S$ of vertices. 
from Lemma~\ref{lem:todos}\ref{lem:nod4insameK4}.

\begin{corollary}\label{cor+lemma:no2d4}
Let~$G$ be a minimum counterexample to the $\alpha$-FC Conjecture, for $2 \leq \alpha \leq 3$.
Then the following hold:
  \begin{enumerate*}[label=(\alph*),itemjoin={\ }]
    \item\label{cor:no2d4neighbors}every degree-4 vertex in~$G$ has at most one degree-4 neighbor; and
    \item\label{lem:dgr5-at-least-two-degree-5-neighbors}each vertex with degree at least 5 in~$G$ has at least two neighbors of degree at least~5.
  \end{enumerate*}
\end{corollary}

Corollary~\ref{cor+lemma:no2d4}\ref{lem:dgr5-at-least-two-degree-5-neighbors} is also a strengthening of a result of Chernyshev et al.~\cite[Claim 3]{chernyshev}.
We conclude this section with the proof of Theorem~\ref{thm:teo}\ref{thm:9-4}.

\begin{proof}[Proof of Theorem~\ref{thm:teo}\ref{thm:9-4}]
Suppose~$G$ is a minimum counterexample to Theorem~\ref{thm:teo}\ref{thm:9-4}, and hence to the $\frac94$-FC Conjecture. 
Let~$n$ be the number of vertices of~$G$, 
and \(n_i\) be the number of degree-\(i\) vertices in~\(G\).
By Lemma~\ref{lem:todos}\ref{lem:4connected}, $G$ is 4-connected and \(n = \sum_{i = 4}^{n-1} n_i \geq 8\).
Let \(F_4\) be the set of edges joining degree-4 vertices to vertices with degree at least 5.
By Corollary~\ref{cor+lemma:no2d4}\ref{cor:no2d4neighbors}, we have that \({|F_4| \geq 3n_4}\).
By Corollary~\ref{cor+lemma:no2d4}\ref{lem:dgr5-at-least-two-degree-5-neighbors}, 
each degree-\(j\) vertex in~$G$ with \(j\geq 5\) contributes with at most \(j-2\) edges to \(F_4\), and hence \(|F_4| \leq \sum_{j = 5}^{n-1} (j-2)n_j\).
Now, since~\(j-2 \leq 6j-27\) for \(j \geq 5\),
we have \( 3n_4 \leq \sum_{j = 5}^{n-1} (j-2)n_j \leq \sum_{j=5}^{n-1} (6j-27)n_j = 6(2e(G) - 4n_4) - 27(n-n_4) = 12e(G) + 3n_4 - 27n
\),
so \(e(G) \geq 9n/4\), a contradiction.
% \begin{align*}
%     2e(G)  & %= \ \sum_{j=4}^{n-1} j\cdot n_j 
%         %\ = \ 4n_4 + \sum_{j=5}^{n-1} j\cdot n_j 
%          \ = \ 4n_4 + \frac{9}{2}\sum_{j=5}^{n-1} n_j + \sum_{j=5}^{n-1} \big(j-\frac{9}{2}\big)\cdot n_j %\\ 
%           \geq \ 4n_4 + \frac{9}{2}\sum_{j=5}^{n-1} n_j + \frac{1}{6}\sum_{j=5}^{n-1} (j-2)\cdot n_j
%         \ \geq \ 
%         %4n_4 + \frac{9}{2}\sum_{j=5}^{n-1} n_j + \frac{n_4}{2} 
%         %\ = \ 
%         %\frac92\sum_{j=4}^{n-1}{n_j} \ = \ 
%         \frac{9}{2}n. \qedhere
% \end{align*}
\end{proof}

\section{Bounds for 1-cyclic and 2-cyclic graphs}\label{sec:1cyclic}\label{sec:2cyclic}

First, we present a family of counterexamples to Conjecture~\ref{conj:construction} and prove Remark~\ref{rem:todas}\ref{rm:3conn-1cyc}. 
%Wormald~\cite{wormald1981} proved that almost all~$r$-regular graphs are~$r$-connected. 
Take any 3-connected 3-regular graph (see~\cite{wormald1981}) with~$k$ vertices and replace each vertex with a~$K_4$, 
connecting each of its neighbors to a distinct vertex in the~$K_4$ 
and leaving only one vertex of each $K_4$ with degree~3 (see, e.g., Figure~\ref{fig:blowupK4Petersen}). 
We obtain a 3-connected graph $G$ with precisely~$n=4k$ vertices and~${m = \frac{3k}{2} + 6k = \frac{15}{8}n}$ edges. 
Moreover, $G$ is \(1\)-cyclic because each of its vertices is in a \(K_4\).
% Each of its vertices has a cycle in its neighborhood because every vertex is in a $K_4$.
%
% \begin{comment}
\begin{figure}[ht]
    \centering
    \scalebox{0.7}{\begin{tikzpicture}[scale = .5]
    \foreach \i in {1,...,3}{
        \node(x\i) [black vertex] at (30+\i*360/3:3.3) {};
    }	
    \foreach \i in {1,...,3}{
	  \node(y\i) [black vertex] at 
        (30+\i*360/3:2.55) {};	
    }
    \foreach \i in {1,...,3}{
	  \node(w\i) [black vertex] at
        (30+11+\i*360/3:3.8) {};	
    }
    \foreach \i in {1,...,3}{
	  \node(z\i) [black vertex] at
        (30-11+\i*360/3:3.8) {};	
    }
 
    \foreach \i in {1,...,3}{
        \draw[edge] (x\i) -- (y\i);
        \draw[edge] (x\i) -- (w\i);
        \draw[edge] (x\i) -- (z\i);
        \draw[edge] (y\i) -- (w\i);
        \draw[edge] (w\i) -- (z\i);
        \draw[edge] (z\i) -- (y\i);
    }

    \draw[edge] (z1) -- (w3);
    \foreach \i in {1,...,2}{
        \pgfmathtruncatemacro{\j}{\i + 1}
        \draw[edge] (w\i) -- (z\j);
    }

    \foreach \i in {1,...,3}{
	  \node(x\i) [black vertex] at
        (30+\i*360/3:0.8) {};	
    }
	\node(x0) [black vertex] at (30:0) {};
 
    \draw[edge] (x0) -- (x1);
    \draw[edge] (x0) -- (x2);
    \draw[edge] (x0) -- (x3);
    \draw[edge] (x1) -- (x2);
    \draw[edge] (x1) -- (x3);
    \draw[edge] (x2) -- (x3);

    \draw[edge] (y1) -- (x1);
    \draw[edge] (y2) -- (x2);
    \draw[edge] (y3) -- (x3);
\end{tikzpicture}}
    \caption{A counterexample to Conjecture~\ref{conj:construction} built from $K_4$.}
    \label{fig:blowupK4Petersen}
\end{figure}
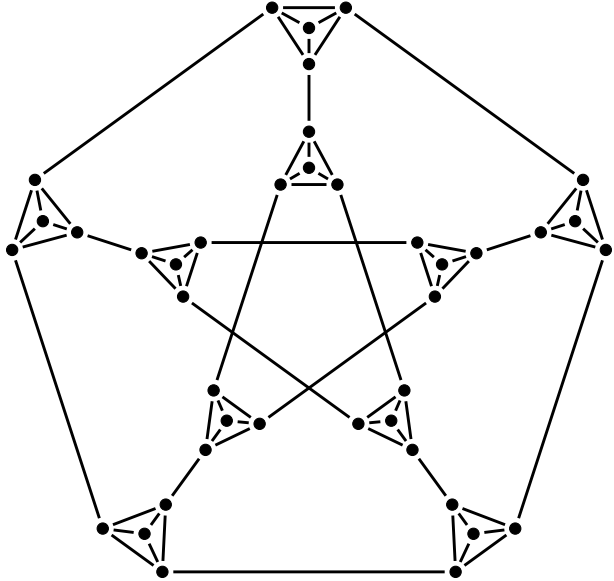
% \end{comment}
%

Remark~\ref{rem:todas}\ref{rm:3conn-1cyc} shows that Theorem~\ref{thm:teo}\ref{thm:15-8} is tight. We denote by \(K^\triangle_s\) the graph obtained from \(K_3\) by adding \(s\) new vertices adjacent to the three vertices of the \(K_3\).  The proof of Theorem~\ref{thm:teo}\ref{thm:15-8} uses the following lemma, whose proof we omit.
%
% \begin{lemma}\label{lem:degree3}
%   If~$G$ is a 3-connected 1-cyclic graph on $n \geq 5$ vertices, then every degree-3 vertex has no degree-3 neighbor. 
% \end{lemma}
% \begin{comment}
% \begin{proof}
%   Let $v$ be a degree-3 vertex, which is not universal in $G$ because $n \geq 5$. Thus, $v$ forms a $K_4$ with its neighbors since $G$ is 1-cyclic. If one of the neighbors of~$v$ has degree~3, the other two neighbors form a 2-vertex cut in~$G$, which contradicts the fact that~$G$ is 3-connected. 
% \end{proof}
% \end{comment}
%
% In what follows denote by \(K^\triangle_s\) the graph obtained from \(K_3\) by adding \(s\) new vertices adjacent to the three vertices of the \(K_3\).
%
% \begin{lemma}\label{lem:3nei}
%   If~$G$ is a 3-connected 1-cyclic graph on $n \geq 5$ vertices, then either~$G$ is isomorphic to $K^\triangle_{n-3}$ or every vertex of \(G\) has at least three neighbors of degree at least~4. 
% \end{lemma}
\begin{lemma}\label{lem:degree3+3nei}
  If~$G$ is a 3-connected 1-cyclic graph on $n \geq 5$ vertices.
  Then the following hold:
  \begin{enumerate*}[label=(\alph*),itemjoin={\ }]
    \item\label{lem:degree3}every degree-3 vertex has no degree-3 neighbor; and 
    \item\label{lem:3nei}either~$G$ is isomorphic to $K^\triangle_{n-3}$ or every vertex of \(G\) has at least three neighbors of degree at least~4. 
  \end{enumerate*}
\end{lemma}
\begin{comment}
\begin{proof}
  Let \(v\) be a degree-\(k\) vertex.
  If \(k = 3\), then the statement follows by Lemma~\ref{lem:degree3}.
  So, we may assume that~\(k \geq 4\).
  Let \(N_3\) be the set of degree-\(3\) neighbors of \(v\),
  and suppose that \(v\) has at most two neighbors with degree at least \(4\).
  For each \(u\in N_3\), let \(N(u) = \{v,x_u,y_u\}\).
  By Lemma~\ref{lem:degree3}, both $x_u$ and $y_u$ have degree at least~4.
  Moreover, \(N(u)\) induces a triangle in \(G\), and thus~$x_u, y_u \in N(v) \setminus N_3$.
  Since \(v\) has at most two neighbors with degree at least \(4\),
  we must have \(\{x_u,y_u\} = \{x_w,y_w\}\) for every \(u,w\in N_3\).
  Therefore,~$G$ is isomorphic to $K^\triangle_{n-3}$.
\end{proof}
\end{comment}

% We are ready to prove Theorem~\ref{thm:teo}\ref{thm:15-8}. 

\begin{proof}[Proof of Theorem~\ref{thm:teo}\ref{thm:15-8}]
 Let~$G$ be a 3-connected 1-cyclic graph on $n \geq 6$ vertices, and~$n_i$ be the number of degree-$i$ vertices in~$G$.   
 By Lemma~\ref{lem:degree3+3nei}\ref{lem:3nei}, 
 either \(G\) is isomorphic to \(K^\triangle_{n-3}\) or every vertex of \(G\) has at least three neighbors of degree at least~4.
 In the former case, as desired, $e(G) = 3n-6>\frac{15}8n$ as~$n \geq 6$.
 In the latter case, as $4j-15 \geq j-3$ for $j \geq 4$, we have
 $3n_3 \leq \sum_{j=4}^{n-1}(j-3)n_j \leq \sum_{j=4}^{n-1}(4j-15)n_j = 8e(G) - 15n + 3n_3$, i.e., $e(G) \geq \frac{15}8n$.
\end{proof}

% A $K_5$ minus an edge is a 3-connected 1-cyclic graph on 5 vertices and 9 edges, but $9 < 15 \cdot \frac58 = \frac{75}8$. This explains the requirement that $n \geq 6$ in Theorem~\ref{thm:teo}\ref{thm:15-8}.

% \medskip

% Observe that any \(4\)-connected \(1\)-cyclic graph on \(n\) vertices 
% has minimum degree \(4\), and hence has at least \(2n\) edges.
% A 4-connected graph on $n$ vertices has at least $2n$ edges. So if~$G$ is a 4-connected 1-cyclic graph on $n \geq 6$ vertices, then $e(G) \geq 2n$. 
Note that if we pick an arbitrary 4-connected 4-regular graph and replace each of its vertices by a $K_4$, leaving all vertices of each $K_4$ with degree~4,
then the graph obtained is 4-connected, 4-regular, and 1-cyclic. 
Therefore, the lower bound \(e(G) \geq 2n\) is best possible for 4-connected 1-cyclic graphs, and proves Remark~\ref{rem:todas}\ref{rm:4conn-1cyc}. 
%
% \section{Bounds for 2-cyclic graphs}\label{sec:2cyclic}
%
Now, we prove a lower bound on the number of edges for a 3-connected graph to be 2-cyclic. Specifically, we prove Theorem~\ref{thm:teo}\ref{thm:2n}. 
We start by proving some properties of 3-connected 2-cyclic graphs. 

\begin{lemma}\label{lem:degree3-2degree5}
  Let~$G$ be a 3-connected 2-cyclic graph on $n\geq 6$ vertices. 
  Then every degree-3 vertex has at least two neighbors of degree at least 5. 
\end{lemma}

% XXX: Usamos N({v, x}) sem definir antes aqui.
\begin{proof}
  Let $v$ be a degree-3 vertex in~$G$, and $x$, $y$, and $z$ be its neighbors. By Lemma~\ref{lem:degree3+3nei}\ref{lem:degree3}, these three vertices have degree at least 4, and they form a triangle, because \(n\geq 5\) and~$G$ is 1-cyclic.  Suppose, for a contradiction, that $x$ and $y$ have degree~4. 
  Then the neighborhood $N(\{v,x\}) = \{y,z,w\}$, where~$w$ is the other neighbor of~$x$. 
  As $n \geq 6$ and~$G$ is 2-cyclic, $y,x,w$ form a triangle, and~$w$ is also the other neighbor of~$y$. But then~$N(\{x,y\}) = \{v, z, w\}$, which must form a cycle because~$n \geq 6$. However there is no edge~$vw$, a contradiction.
\end{proof}

% We proceed with the proof of Theorem~\ref{thm:teo}\ref{thm:2n}.

\begin{proof}[Proof of Theorem~\ref{thm:teo}\ref{thm:2n}]
  Let $n_i$ be the number of degree-$i$ vertices in~$G$ and \(F\) be the set of edges joining degree-3 vertices to vertices with degree at least~5. 
  %Let us estimate the size of~\(F\).
  By Lemma~\ref{lem:degree3-2degree5}, we have that~\(|F| \geq 2 n_3\).
  By Lemma~\ref{lem:degree3+3nei}\ref{lem:3nei}, 
  either \(G\) is isomorphic to \(K^\triangle_{n-3}\) or every vertex of \(G\) has at least three neighbors of degree at least~4.
  In the former case, $G$ has $3n-6 \geq 2n$ edges as~$n \geq 6$.
  In the latter case, each degree-\(j\) vertex for $j \geq 5$ contributes with at most \(j-3\) edges to $F$,
  so \(|F|\leq \sum_{j = 5}^{n-1} (j-3)n_j\).
  %Therefore \(\sum_{j = 5}^{n-1} \frac{j-3}{2}n_j \geq n_3\). 
  As \(2j-8 \geq j-3\) for \(j\geq 5\), we have
  \(
    2n_3 \leq |F| \leq \sum_{j = 5}^{n-1} (j-3)n_j
    \leq \sum_{j = 5}^{n-1} (2j-8)n_j
 %   = 2(2e(G) - 4n_4 - 3n_3) - 8(n-n_4-n_3)
    = 4e(G) - 8n + 2n_3
  \),
  i.e., \(e(G) \geq 2n\).
  % as desired.
  % \begin{align*}
  %   2m \ = \ \sum_{j = 3}^{n-1} j\cdot n_j 
  %       & \ = \ 3n_3 +  \sum_{j = 4}^{n-1} j\cdot n_j 
  %       \ = \ 3n_3 + 4\sum_{j = 4}^{n-1} n_j + \sum_{j = 5}^{n-1} (j-4)\cdot n_j \\
  %       &\ \geq \ 3n_3 + 4\sum_{j = 4}^{n-1} n_j + \sum_{j = 5}^{n-1} \frac{j-3}{2}\cdot n_j 
  %       \ \geq \ 3n_3 + 4\sum_{j = 4}^{n-1} n_j + n_3 
  %       %\ = \ 4\sum_{j=3}^{n-1}n_j 
  %       \ = \ 4n. \qedhere
  %   \end{align*}
\end{proof}

In Figure~\ref{fig:2n-edges}, on the left, we show two tight examples  for Theorem~\ref{thm:teo}\ref{thm:2n}: the graph $K^\triangle_3$, which is 3-connected, 
and the octahedral graph, which is 4-connected. 
The third graph in Figure~\ref{fig:2n-edges} has $9$ vertices and $20$ edges.
Consider the construction illustrated in Figure~\ref{fig:blowupK4Petersen}, starting from a 3-connected 3-regular graph on $k$ vertices. If we replace each vertex by an octahedral graph instead of a $K_4$, we end up with a 3-connected 2-cyclic graph on $6k$ vertices and~$\frac32k + 12k = \frac{27}2k = \frac94n$ edges, which proves Remark~\ref{rem:todas}\ref{rm:3conn-2cyc}.
As far as we know, it may hold that~$m \geq \frac94n$ for the graphs addressed by Theorem~\ref{thm:teo}\ref{thm:2n} if~$n \geq 10$. The requirement~$n \geq 10$ is necessary to exclude the third graph in Figure~\ref{fig:2n-edges}, because~$\frac{20}{9} < \frac94$.

The lower bound on the number of edges in a 4-connected 2-cyclic graph might be larger. 
Take a 4-connected 4-regular graph on~$k$ vertices,
and replace each of its vertices by an octahedral graph, leaving precisely four vertices of each octahedral graph with degree~5. 
The graph obtained is 4-connected, 2-cyclic, has~$6k$ vertices and~$m = 2k+12k=14k = \frac73n$ edges.
This proves Remark~\ref{rem:todas}\ref{rm:4conn-2cyc},
which shows that Conjecture~\ref{conj:73} is tight. 
In Figure~\ref{fig:2n-edges}, on the right, we show a 4-connected 2-cyclic graph on $7$ vertices and $16$ edges, and two 4-connected 2-cyclic graphs with $8$ vertices and $18$ edges. 
Since $\frac{16}7$ and $\frac{18}8$ are less than $\frac73$, these examples justify the condition $n \geq 9$ in Conjecture~\ref{conj:73}.

% This, and the family of graphs described next, proving Remark~\ref{rem:todas}\ref{rm:4conn-2cyc}, led us to propose Conjecture~\ref{conj:73}, which would imply a better result towards Conjecture~\ref{conj:3n}.
\begin{figure}[ht]
    \centering
    \scalebox{0.7}
    {% \resizebox{!}{3cm}{
\begin{tikzpicture}[scale = .5]
% K*3: um triangulo e três vértices independentes conectados aos três vértices do triângulo. 
    \foreach \i in {1,...,3}{
	\node(t\i) [black vertex] at (90+\i*360/3:1) {};	
    }
    \foreach \i in {1,...,3}{
	\node(i\i) [black vertex] at (90+\i*360/3:3) {};
    }	
    \draw[edge] (t1) -- (t2);
    \draw[edge] (t2) -- (t3);
    \draw[edge] (t3) -- (t1);
    \draw[edge] (i1) -- (t1);
    \draw[edge] (i1) -- (t2);
    \draw[edge] (i1) -- (t3);
    \draw[edge] (i2) -- (t1);
    \draw[edge] (i2) -- (t2);
    \draw[edge] (i2) -- (t3);
    \draw[edge] (i3) -- (t1);
    \draw[edge] (i3) -- (t2);
    \draw[edge] (i3) -- (t3);

% Octahedron
    \node(bot) [black vertex] at  (7,-1.5) {};
    \node(top) [black vertex] at  (7,3.5) {};
    \node(lef) [black vertex] at (5,0.5) {};
    \node(rig) [black vertex] at (9,1.5) {};
    \node(bac) [black vertex] at (7.5,0.5) {};
    \node(fro) [black vertex] at (6.5,1.5) {};
    \draw[edge] (bot) -- (lef);
    \draw[edge] (bot) -- (bac);
    \draw[edge] (bot) -- (rig);
    \draw[edge] (bot) -- (fro);
    \draw[edge] (lef) -- (fro);
    \draw[edge] (rig) -- (bac);
    \draw[edge] (top) -- (lef);
    \draw[edge] (top) -- (bac);
    \draw[edge] (top) -- (rig);
    \draw[edge] (top) -- (fro);
    \draw[edge] (lef) -- (bac);
    \draw[edge] (rig) -- (fro);
% Segundo grafo
%   \node(x1) [black vertex] at  (6.5,2) {}; % left
%   \node(x4) [black vertex] at (11.5,2) {}; % right
%   \node(x6) [black vertex] at    (8,1) {}; % bottom l
%   \node(x5) [black vertex] at   (10,1) {}; % bottom r
%   \node(x2) [black vertex] at    (8,3) {}; % top l
%   \node(x3) [black vertex] at   (10,3) {}; % top r
%   \node(uni) [black vertex] at   (9,5) {};
%   \foreach \i in {1,...,6}{
%       \draw[edge] (uni) -- (x\i);
%   }
%   \draw[edge] (x2) -- (x6);
%   \draw[edge] (x3) -- (x5);
%   \draw[edge,red] (x6) -- (x1);
%   \foreach \i in {1,...,5}{
%       \pgfmathtruncatemacro{\j}{\i + 1}
%       \draw[edge] (x\i) -- (x\j);
%   }        
  \foreach \i in {1,...,3}{
    \foreach \j in {1,...,3}{
      \node(x\i\j) [black vertex] at (10+2*\i,-2+2*\j-1) {};
    }
    \foreach \j in {1,...,2}{
      \pgfmathtruncatemacro{\jj}{\j + 1}
      \draw[edge] (x\i\j) -- (x\i\jj);
    }
  }
  \foreach \j in {1,...,3}{
    \foreach \i in {1,...,2}{
      \pgfmathtruncatemacro{\ii}{\i + 1}
      \draw[edge] (x\i\j) -- (x\ii\j);
    }
  }
  \foreach \j in {1,...,2}{
    \foreach \i in {1,...,2}{
      \pgfmathtruncatemacro{\ii}{\i + 1}
      \pgfmathtruncatemacro{\jj}{\j + 1}
      \draw[edge] (x\i\j) -- (x\ii\jj);
      \draw[edge] (x\i\jj) -- (x\ii\j);
    }
  }
\end{tikzpicture}
%}
}
    \hfill
    \scalebox{0.68}
    {% \resizebox{!}{3cm}{
\begin{tikzpicture}[scale = .45]
    % 4-connected, 2-cyclic n=7, m=16
    \foreach \i in {1,...,3}{
	\node(t\i) [black vertex] at (90+\i*360/3:3) {};	
    }
    \foreach \i in {1,...,4}{
	\node(i\i) [black vertex] at (135+\i*360/4:.9) {};
    }	
    \draw[edge] (t1) -- (t2);
    \draw[edge] (t2) -- (t3);
    \draw[edge] (t3) -- (t1);
    \draw[edge] (i1) -- (i2);
    \draw[edge] (i2) -- (i3);
    \draw[edge] (i3) -- (i4);
    \draw[edge] (i4) -- (i1);
    \draw[edge] (i1) -- (t1);
    \draw[edge] (i1) -- (t2);
    \draw[edge] (i2) -- (t2);
    \draw[edge] (i2) -- (t1);
    \draw[edge] (i3) -- (t3);
    \draw[edge] (i3) -- (t2);
    \draw[edge] (i4) -- (t3);
    \draw[edge] (i4) -- (t1);
    \draw[edge] (i4) -- (i2);

    % 4-connected 2-cyclic n=8, m=18
  \begin{scope}[shift={(8.5,1)}]
    \foreach \i in {1,...,4}{
	\node(t\i) [black vertex] at (\i*360/4:3) {};	
    }
    \foreach \i in {1,...,4}{
	\node(i\i) [black vertex] at (\i*360/4:1) {};
    }	
    \draw[edge] (t1) -- (t2);
    \draw[edge] (t2) -- (t3);
    \draw[edge] (t3) -- (t4);
    \draw[edge] (t4) -- (t1);
    \draw[edge] (i1) -- (i2);
    \draw[edge] (i2) -- (i3);
    \draw[edge] (i3) -- (i4);
    \draw[edge] (i4) -- (i1);
    \draw[edge] (i1) -- (t1);
    \draw[edge] (i2) -- (t2);
    \draw[edge] (i3) -- (t3);
    \draw[edge] (i4) -- (t4);
    \draw[edge] (i1) -- (i3);
    \draw[edge] (t1) -- (i4);
    \draw[edge] (t1) -- (i2);
    \draw[edge] (t3) -- (i4);
    \draw[edge] (t3) -- (i2);
    \draw[edge, bend right=120] (t2) to (t4);
  \end{scope}

  % 4-connected 2-cyclic n=8, m=18
  \begin{scope}[shift={(18, 1)}]
    \foreach \i in {1,...,6}{
	\node(c\i) [black vertex] at (\i*360/6:3) {};	
    }
	\node(a1) [black vertex] at ($(c1)!.6!(c5)$) {};
	\node(a2) [black vertex] at ($(c2)!.4!(c4)$) {};	
	\draw[edge] (c1) -- (c2);
    \draw[edge] (c2) -- (c3);
    \draw[edge] (c3) -- (c4);
    \draw[edge] (c4) -- (c5);
    \draw[edge] (c5) -- (c6);
    \draw[edge] (c6) -- (c1);
    \draw[edge] (a1) -- (c1);
    \draw[edge] (a1) -- (c2);
    \draw[edge] (a1) -- (c3);
    \draw[edge] (a1) -- (c4);
    \draw[edge] (a1) -- (c5);
    \draw[edge] (a1) -- (c6);
    \draw[edge] (a2) -- (c1);
    \draw[edge] (a2) -- (c2);
    \draw[edge] (a2) -- (c3);
    \draw[edge] (a2) -- (c4);
    \draw[edge] (a2) -- (c5);
    \draw[edge] (a2) -- (c6);	
  \end{scope}
\end{tikzpicture}
%}
}
    \caption{Left: Three 3-connected 2-cyclic graphs, two with 6 vertices and 12 edges and one with 9 vertices and 20 edges.
    Right: Three 4-connected 2-cyclic graphs, one with 7 vertices and 16 edges, and two with 8 vertices and~18 edges.}
    \label{fig:2n-edges}
\end{figure}

\section{Final remarks}

Several questions remain open. 
Of course it would be nice to settle Conjecture~\ref{conj:3n}, 
or to obtain an improvement on Theorem~\ref{thm:teo}\ref{thm:9-4}. 
Proving Conjecture~\ref{conj:73} or finding a family of 4-connected 2-cyclic graphs on $n$ vertices with less than $\frac73n$ edges would also be interesting. 

The study of $k$-cyclic graphs with $k$ more than~2 seems to be a possible way to achieve better results towards Conjecture~\ref{conj:3n}.
Our exposition points out that we barely use the forest cut requirement for sets larger than~2 in the current results.

\bibliography{main.bib}

\begin{thebibliography}{1}

\bibitem{chen-yu}
Guantao Chen and Xingxing Yu.
\newblock A note on fragile graphs.
\newblock {\em Discrete mathematics}, 249(1-3):41--43, 2002.

\bibitem{chernyshev}
Vsevolod Chernyshev, Johannes Rauch, and Dieter Rautenbach.
\newblock Forest cuts in sparse graphs.
\newblock {\em Available at arXiv:2409.17724}, 2024.

\bibitem{li2024minimum}
Chengli Li, Yurui Tang, and Xingzhi Zhan.
\newblock The minimum size of a 3-connected locally nonforesty graph.
\newblock {\em Available at arXiv:2410.23702}, 2024.

\bibitem{wormald1981}
Nicholas~C. Wormald.
\newblock The asymptotic connectivity of labelled regular graphs.
\newblock {\em Journal of Combinatorial Theory, Series B}, 31(2):156--167, 1981.

\end{thebibliography}
\end{document}